\documentclass[letterpaper,10pt,reqno,final]{amsart}
\usepackage{amsmath}%
\usepackage{amsfonts}%
\usepackage{amssymb}%
\usepackage{graphicx}
\usepackage{functan}
\usepackage{mathrsfs}
\usepackage[all,cmtip]{xy}
\usepackage[obeyspaces]{url}
\newtheorem{theorem}{Theorem}

\newtheorem{corollary}{Corollary}

\newtheorem{definition}{Definition}
\newtheorem{example}{Example}

\newtheorem{lemma}{Lemma}

\newtheorem{remark}{Remark}

\numberwithin{equation}{section}
\numberwithin{theorem}{section}
\numberwithin{lemma}{section}
\numberwithin{corollary}{section}
\numberwithin{definition}{section}
\numberwithin{example}{section}
\numberwithin{remark}{section}
\begin{document}

\title[Nonlinear Evolution in Particular C*-Algebras of Operators]{On the Approximation of Nonlinear Evolution Equations in Particular C*-Algebras of Operators}
\author{Fredy Vides
}
\address
{Escuela de Matem{\'a}tica y Ciencias de la Computaci{\'o}n \newline%
\indent Universidad Nacional Aut{\'o}noma de Honduras}%
\email{fvides@unah.edu.hn}
\urladdr{http://fredyvides.6te.net}
\thanks{This research has been performed in part thanks to the financial support of the School of Mathematics and Computer Science of the National University of Honduras.}
\date{\today}

\begin{abstract}
In this article we deal with the approximation of solutions of nonlinear evolution equations of the form $A(u(t))+f(u(t))=u'(t)$, the numerical analysis of solutions to this problems will be performed using some methods from particular algebras of operators which are sometimes represented by unital subalgebras of the unital C*-algebras of operators that are generated by some basic operators say $\mathbf{1},a,\mathcal{D}(\cdot)\in\mathcal{L}(H^m(G))$ that in some suitable sense are related to the operator $A(\cdot)\in\mathcal{L}(H^m(G))$ in the evolution equations, particular cases where the operator algebras do not verify the C*-identity with respect to the norm chosen are also studied, when applicable basic C*-algebra techniques are implemented to perform some estimates of numerical solutions to some types of problems, in all this work expressions like $H^m(G)$ will represent a prescribed discretizable Hilbert space with $G\subset\subset\mathbb{R}^n$.
\end{abstract}
\maketitle

\section{Introduction}

In this work we will focus our attention in nonlinear evolution equations of the form:
\begin{equation}
\left\{
\begin{array}{l}
A(u(t))+f(u(t))=u'(t) \\
u(0)=u_0, u_0\in H^m(G)
\end{array}
\right .
\label{eq1}
\end{equation}
with $u\in H^m(G_\tau)$ and where $H^m(G),H^m(G_\tau)$ are prescribed Hilbert spaces with $G\subset\subset\mathbb{R}^n$ and $G_\tau:=G\times[0,T], T>0$. In the last equation we have that $A(\cdot)\in\mathcal{L}(h^n(G))$ is an operator variable in time that represents in some suitable sense the boundary conditions of the the initial boundary value problem described by \eqref{eq1}.

We will call any C*-subalgebra $\mathscr{A}(H^n(G))$ of the unital C*-algebra $\mathscr{A}(H^n(G))$ of operators generated by $\mathbf{1},a\in\mathcal{L}(H^m(G))$ such that $\mathbf{1},A(\cdot)\in \mathscr{A}(H^m(G))$, a particular C*-algebra.

In this work, to obtain discrete representations of the operators involved in the semilinear boundary value problems, we will use an abstract procedure known as particular representation method, that is based in the ideas presented by Steinberg and Robidoux in \cite{Steinberg2004} and \cite{SteinRob}, and is applicable to compatible discretization techniques.

\section{Background}

In this section we will present some important concepts that are needed to deduce many of the results presented in this work, the first concepts presented will be those related with sesquilinear forms of operators in prescribed Hilbert spaces $H^n(G)$.

If in a prescribed Hilbert space $H^n(G)$ we have that the inner product $\scalprod*{H^n(G)}{\cdot}{\cdot}$ is related to a map $\mathcal{M}[\cdot](\cdot):H^n(G)\times H^n(G)\rightarrow\mathbb{C}$ by the expression:
\begin{equation}
\scalprod*{H^n(G)}{x}{y}:=\mathcal{M}[x](y)
\end{equation}
then for any operator $A\in \mathscr{A}(H^n(G))$ with $\mathscr{A}(H^n(G))$ a particular C*-algebra of operators we can define the related sesquilinear form by.

\begin{definition}Sesquilinear form of an operator. For any operator $A\in \mathscr{A}$ with $\mathscr{A}:=\mathscr{A}(H^n(G))$ a particular C*-algebra of operators in a prescribed Hilbert space $H^n(G)$ the form $\mathcal{A}[x](y)$ defined by
\begin{equation}
\mathcal{A}[x](y):=\mathcal{M}[Ax](y)=\scalprod*{H^n(G)}{Ax}{y}
\end{equation}
will be called sesquilinear form of the operator $A$.
\end{definition}

\subsection{Contractions}
In this section we will present some important definitions and results related with contractions wich are operators described by the following

\begin{definition} An operator $T\in\mathcal{B}(X)$ with $X$ a Banach space, for wich 
\begin{equation}
\norm{}{T(x)-T(y)}\leq\norm{}{x-y}, x,y\in X 
\end{equation}
is called a contraction. If there is a $K<1$ for wich $\norm{}{T(x)-T(y)}\leq K\norm{}{x-y}$, T is called a strict contraction.
\end{definition}

\begin{theorem}\label{contraction}Contraction Mapping Principle. A strict contraction $T\in\mathcal{B}(X)$ on a Banach space $X$ has a unique fixed point, ie., there exists a unique $x\in X$ such that $T(x)=x$.
\end{theorem}
\begin{proof}
First let us prove uniqueness. If $T(x)=x, T(y)=y$, then $\norm{}{x-y}=\\\norm{}{T(x)-T(y)}\leq K\norm{}{x-y}$. Since $K<1$ and $\norm{}{x-y}\geq 0$, we conclude $\norm{}{x-y}=0$, i.e. $x=y$. To prove existence, we first note that $T$ is automatically continuous since $\norm{}{x-y}<K^{-1}\varepsilon$ implies $\norm{}{T(x)-T(y)}<\varepsilon$. Now, let $x_0$ be arbitrary and let $x_n:=T^n(x_0)$. We will show that $\{x_n\}$ is Cauchy.
\begin{eqnarray}
\norm{}{x_n-x_{n+1}}&=&\norm{}{T(x_{n-1})-T(x_n)}\leq K\norm{}{x_{n-1}-x_n}\\
&\leq&K^2\norm{}{x_{n-2}-x_{n-1}}\\
&\cdots&\leq K^{n-1}\norm{}{x_1-x_0}
\end{eqnarray}
Thus if $n>m$,
\begin{eqnarray}
\norm{}{x_n-x_m}&\leq&\sum_{j=m+1}^{n}\norm{}{x_j-x_{j-1}}\label{approx1}\\
&\leq&K^m(1-K)^{-1}\norm{}{x_0-x_1}\conv{}{n\to\infty}0\label{approx2}
\end{eqnarray}
Thus ${x_n}$ is Cauchy, so $x_n\conv{}{}x$ for some x. Since T is continuous, $T(x)=\lim Tx_n=\lim x_{n+1}=x$ wich proves the theorem.
\end{proof}

From expressions \eqref{approx1} and \eqref{approx2} in the proof presented here, taking $n\to \infty$, we can obtain the following.
\begin{remark}\label{rcontraction}
A fixed point $x\in X$ of a contraction $T\in\mathcal{B}(X)$ in a Banach space $X$, satisfies the following estimate
\begin{equation}
\norm{}{x-x_m}\leq K^m(1-K)^{-1}\norm{}{x_1-x_0}.
\end{equation}
\end{remark}

\subsection{Elliptic Operators}

A differential operator defined for any prescribed $u\in H^2(G)$ with $G\subset\subset\mathbb{R}^n$ in the form:
\begin{equation}
Au:=\sum_{i,j}\partial_i (a_{i,j}(x)\partial_j u)+\sum_j a_j(x)\partial_j u, \:\: x\in G
\label{elliptic}
\end{equation}
with $\partial_i:=\partial_{x_i}$, can be expressed in the following sesquilinear form

\begin{equation}
\mathcal{A}[u](v):=\int_G \{\sum_{i,j}(a_{i,j}\partial_j u)\partial_i \overline{v}+\sum_j a_j\partial_j u\cdot \overline{v}\}, \:\: x\in G
\label{elliptic1}
\end{equation}
that will be related to uniformly elliptic equations by the following.

\begin{definition}Strongly elliptic sesquilinear form. A sesquilinear form like \eqref{elliptic1} is said to be strongly elliptic if there is a constant $c_0>0$ such that
\begin{equation}
Re\sum_{i,j}a_{i,j}(x)\xi_i \overline{\xi}_j \geq c_0 \sum_j |\xi_j|^2 \:\: ,x\in G,\xi\in\mathbb{C}^n.
\label{elliptic2}
\end{equation}
\end{definition}

If $\mathscr{A}(H^2(G))(G)$ is a particular C*-algebra on a Hilbert space $H^n(G)$, the sesquilinear form $\mathcal{L}$ of a uniformly elliptic operator $L\in \mathscr{A}(H^2(G))(G))$ can be expressed in the following way
\begin{equation}
\mathcal{L}[u](v)=\int_G\{\sum_{i,j}a_{i,j}\partial_iu\cdot\partial_j\overline{v}+\sum_j a_j\partial_ju\cdot\overline{v} \}
\label{seski}
\end{equation}
sesquilinear forms of this type are said to be

\begin{definition}Strongly Elliptic Form. A sesquilinear form like \eqref{seski} is said to be strongly elliptic if there is a $c_L>0$ such that
\begin{equation}
Re\sum_{i,j}a_{i,j}\xi_i\overline{\xi}_j\geq c_L\sum_j|\xi_j|^2, \:\:\: x\in G, \xi\in\mathbb{C}^n.
\end{equation}
\end{definition}

Now we will introduce some important definitions related to the well posedness of boundary value problems and that generalize the above results to Hilbert spaces $H^n(G)$ and particular C*-algebras $\mathscr{A}(H^n(G))$ of operators.

\begin{definition}
Coercivity of a form. The sesquilinear form $\mathcal{A}[\cdot](\cdot)$ of an operator $A\in \mathscr{A}(H^n(G))$ with $\mathscr{A}(H^n(G))$ a particular C*-algebra over a Hilbert space $H^n(G)$ is said to be $H^n(G)$-coercive if there is a $c_A>0$ such that
\begin{equation}
|\mathcal{A}[x](x)|\geq c_A\norm*{H^n(G)}{x}^2, x\in H^n(G)
\end{equation}
with $\norm*{H^n(G)}{x}:=\scalprod*{H^n(G)}{x}{x}^{1/2}$.
\end{definition}

\begin{definition} Ellipticity of a form. The sesquilinear form $\mathcal{A}[\cdot](\cdot)$ of an operator $A\in \mathscr{A}(H^n(G))$ with $\mathscr{A}(H^n(G))$ a particular C*-algebra on a Hilbert space $H^n(G)$ is said to be $H^n(G)$-elliptic if there is a $c_A>0$ such that
\begin{equation}
Re \mathcal{A}[x](x) \geq c_A\norm*{H^n(G)}{x}^2, x\in H^n(G).
\end{equation}
\end{definition}

\begin{definition} Bounded Operator. A symmetric $V$-elliptic operator $A\in\mathscr{A}(V)$ with $H^n(G)\geqslant V$ is said to be $V$-bounded if there is a constant $0<K_A<\infty$ such that $|\mathcal{A}[u](v)|\leq K_A \norm*{V}{u}\norm*{V}{v}$, $u,v\in V$.
\end{definition}

An operator $A\in\mathscr{A}(H^n(G))$ is said to be $V$-elliptic with $H^n(G)\geqslant V$ if its sesquilinear form $\mathcal{A}[\cdot](\cdot)$ is $V$-elliptic, in a similar way if a sesquilinear form $\mathcal{A}[\cdot](\cdot)$ is $V$-coercive then the related operator $A\in\mathscr{L}(H^n(G))$ is said to be $V$-coercive. If for a given Hilbert space $H^n(G)$ and a prescribed subspace $H^n(G)\geqslant V$ we have a symmetric $V$-elliptic operator $A\in \mathscr{A}(H^n(G))$ with $\mathscr{A}(H^n(G))$ a particular C*-algebra generated by a given operator $a \in \mathscr{L}(H^n(G))$, and if we have that $A=a^\dagger a$, it is not very difficult to see that
\begin{eqnarray}
\mathcal{A}[x](y)&=&\scalprod*{H^n(G)}{Ax}{y}\\
&=&\scalprod*{H^n(G)}{a^\dagger ax}{y}\\
&=&\scalprod*{H^n(G)}{ax}{ay}\\
&=&\scalprod*{a}{x}{y}\label{aprod}
\end{eqnarray}
the equation \eqref{aprod} is often called $a$-inner product and is useful to perform several estimates related with this kind of operators. A particularly important example of this type of inner products is the corresponding to the laplacian operator $\Delta:=\nabla\cdot\nabla=\sum_k \partial_k^2$ with $\nabla=:grad=[\partial_1,\cdots,\partial_n]$ and $\nabla\cdot=:div$ in particular function spaces, one of this examples will be presented now.

The importance of previous properties will be remarked in the following

\begin{lemma} Lax-Milgram Lemma. For a given operator $A\in \mathscr{A}(H)$ with $H$ a prescribed Hilbert space
if we have that $A$ is $V$-coercive and $V$-bounded, with $H \leqslant V$, then for each given $f\in V$, there exists a unique $u\in V$ such that
\begin{equation}
\mathcal{A}[u](v)=\scalprod*{V}{f}{v}, \:\: v\in V
\label{weak}
\end{equation}
Furthermore, if for a given $f$ the solution $u$ of \eqref{weak} is denoted by
\begin{equation}
u=\mathcal{G}f
\label{green}
\end{equation}
then $G\in\mathcal{L}(V)$.
\end{lemma}
\begin{proof}
For any $u,v\in V$ and $f\in H$, with $V\leqslant H$, it can be seen that $\mathcal{A}[\cdot](\cdot)$ induces a scalar product in the following way
\begin{equation}
\scalprod*{V}{u}{v}:=\mathcal{A}[u](v), u,v \in V
\end{equation}
and also we can write $f(v):=\scalprod*{H}{f}{v}$ since $A$ is $V$-bounded and
\begin{equation}
\scalprod*{V}{u}{v}=f(v), v\in V
\end{equation}
by the Riesz representation theorem there exists $u\in V$ such that $\scalprod*{V}{u}{v}=f(v), \forall v\in V$, hence there exists a solution to \eqref{weak}, on the other hand we have that if there are two such solutions $u_1,u_1 \in V$ then
\begin{equation}
\mathcal{A}[u_1](v)=\scalprod*{H}{f}{v}=\mathcal{A}[u_1](v), v\in V
\end{equation}
wich implies that
\begin{equation}
\mathcal{A}[u_1-u_2](v)=0, v\in V
\end{equation}
since this is true for each $v\in V$ taking $v=u_1-u_2$ we get
\begin{equation}
0=\mathcal{A}[u_1-u_2](u_1-u_2)\geq c_A \norm*{V}{u_1-u_2}^2\geq 0
\end{equation}
wich implies that $\norm*{V}{u_1-u_2}=0$ hence $u_1=u_2$. Uniqueness allows us to define a map $\mathcal{G}\in \mathcal {L}(V)$ such that for a given $f\in V$ we have that $\mathcal{G}f$ is the unique element of $V$ that satisfies
\begin{equation}
\mathcal{A}[\mathcal{G}f](v)=\scalprod*{V}{f}{v}, v\in V
\end{equation}
clearly for $f_1,f_2\in V$ and scalars $a_1,a_2\in \mathbb{C}$ we will have that for every $v\in V$
\begin{eqnarray}
\mathcal{A}[a_1\mathcal{G}f_1+a_2\mathcal{G}f_2](v)&=&a_1\mathcal{A}[\mathcal{G}f_1](v)+a_2\mathcal{A}[\mathcal{G}f_2](v)\\
&=&a_1\mathcal{A}[f_1](v)+a_2\mathcal{A}[f_2](v)\\
&=&\mathcal{A}[a1f_1+a_2f_2](v)
\end{eqnarray}
Thus $\mathcal{G}(a_1f_1+a_2f_2)$ is solution of 
\begin{equation}
\mathcal{A}[y](v)=\scalprod*{V}{a_1f_1+a_2f_2}{v}=\mathcal{A}[a_1\mathcal{G}f_1+a_2\mathcal{G}f_2](v), v\in V
\end{equation}
hence $G\in\mathcal{L}(V)$.
\end{proof}

\begin{example}
If we take $A=-\Delta_{H^2_0(G)}:=-\Delta|_{H^2_0(G)}$ with $G:=[-1,1]^3$ and with $a:=\nabla$ and $a^\dagger:=\nabla\cdot$ and if we take the scalar product $\scalprod*{L^2(G)}{u}{v}$ to be defined by
\begin{equation}
\scalprod*{G}{u}{v}:=\int_G u\overline{v}
\end{equation}
by the first Green's formula we can obtain that
\begin{eqnarray}
\mathcal{A}[u](v)&:=&\scalprod*{L^2(G)}{Au}{v}\\
&=&-\int_G \Delta_{H^2_0(G)}u\overline{v}\\
&=&\int_G\nabla u \cdot \nabla \overline{v}\\
&=&\scalprod*{a}{u}{v}
\end{eqnarray}
now we can see that by the Poincare inequality there is a number $c(G)>0$ that depends on the Lebesgue measure of $G$ such that
\begin{eqnarray}
\norm*{L^2(G)}{u}&\leq&c(G)\int_G \nabla u \cdot\nabla \overline{u}\\
&=&c(G)\scalprod*{L^2(G)}{\nabla u}{\nabla u}\\
&=&c(G)|\mathcal{A}[u](u)|\\
&=&c(G)\norm*{a}{u}^2
\end{eqnarray}
wich implies that
\begin{eqnarray}
\norm*{H^1(G)}{u}^2&=&\norm*{L^2(G)}{u}^2+\norm*{L^2(G)}{\nabla u}^2\\
&\leq& c(G)^2\norm*{L^2(G)}{\nabla u}^2+\norm*{L^2(G)}{\nabla u}^2\\
&=&(1+c(G)^2)\norm*{L^2(G)}{\nabla u}^2\\
&=&(1+c(G)^2)\norm*{a}{u}^2\\
&=&(1+c(G)^2)|\mathcal{A}[u](u)|
\end{eqnarray}
wich means that there is a $c_A=(1+(c(G))^2)^{-1}>0$ for all $u\in H^2_0(G)$ such that $\mathcal{A}[u](u)\geq c_A\norm*{H^1(G)}{u}^2$ hence we have that $A$ is $H^1(G)$-coercive. On the other hand using the Schwarz inequality we can check that for $u,v\in H^2_0(G)$ we will have
\begin{eqnarray}
|A[u](v)|&=& \scalprod*{L^2(G)}{au}{av}\\
&\leq& \norm*{L^2(G)}{au} \norm*{L^2(G)}{av}\\
&\leq& \left[\norm*{L^2(G)}{au}^2+\norm*{L^2(G)}{u}^2\right]^{1/2} \left[\norm*{L^2(G)}{av}^2+\norm*{L^2(G)}{u}^2\right]^{1/2}\\
&=&\norm*{H^1(G)}{u}\norm*{H^1(G)}{v}
\end{eqnarray}
therefore $A$ is $H^1_0(G)$-bounded.
\end{example}

\subsection{Parabolic Operators}

A differential operator $P\in\mathscr{A}(H^n(G_\tau))$ defined by:
\begin{eqnarray}
Pu&:=&Au-\partial_t u\\
&=&\sum_{i,j}\partial_i (a_{i,j}(x,t)\partial_j u)+\sum_j a_j(x,t)\partial_j u -\partial_t u
\end{eqnarray}
with $(x,t)\in G_\tau= G\times[t_0,T]$, and where $A$ is $V$-elliptic for $ H^n(G_\tau) \leqslant V$, is said to be parabolic. In this section we will introduce an important condition trough the following.

\begin{definition} Dissipative Operator. An elliptic operator $A\in\mathscr{L}(H^n(G))$ is said to be dissipative if we have that
\begin{equation}
\text{Re} \mathcal{A}[x](x)\leq 0 \:\:, x\in \text{dom}(A).
\end{equation}
\end{definition}

The solution to a parabolic equation can be expressed using a general concept presented in the following definition.

\begin{definition} Contractive Semigroups. A set $\{s_t:t\geq 0\}$ of operators on $H^n(G)$ that satisfy the following conditions
\begin{description}
\item[S1] $\norm{}{s_t(x)}\leq\norm{}{x}$, $x\in\text{dom}(A), t\geq 0$
\item[S2] $s_{t+\tau}(\cdot)=s_t\circ s_\tau(\cdot),s_0=\mathbf{1}$, $t,\tau\geq 0$,
\item[S3] $s_{(\cdot)}(x)\in C([0,\infty)),H^n(G))$, $x\in H^n(G)$.
\end{description}
called contractive semigroup conditions.
\end{definition}

A semigroup $\{s_t:t\geq 0\}$ is said to be generated by an elliptic operator $A\in\mathscr{A}(H^n(G))$ if we have that
\begin{equation}
\lim_{t\to 0^+}h^{-1}(s_h(\cdot)-\mathbf{1})=A(\cdot)
\end{equation} 

If we have a dissipative elliptic operator $A\in\mathscr{A}(H^n(G_\tau))$ a solution $v\in H^n(G_\tau)$ to the Cauchy problem
\begin{equation}
\left\{
\begin{array}{l}
Av=\partial_t v, (x,t)\in G_\tau=G\times [0,T] \\
v(x,0)=v_0,x\in G\\
Bv=v_b, x\in \partial G
\end{array}
\right .
\end{equation} 
will satisfy
\begin{eqnarray}
D_t(\norm{}{v(x,t)})&=&2\text{Re}\scalprod*{H^n(G_\tau)}{\partial_tv(x,t)}{v(x,t)}\\
&=&2\text{Re}\mathcal{A}[v](v)\leq 0, \:\: t>0,
\end{eqnarray}
so it follows that $\norm{}{v(x,t)}\leq\norm{}{v(x,0)}$, $t\geq 0$. This shows that
\begin{equation}
\norm{}{s_t v_0}\leq\norm{}{v_0}, v_0\in dom(A), t\geq 0,
\end{equation}
so each $s_t$ is a contraction in the $H^n(G)$-norm.

The relations presented above can be sumarized in the following result.

\begin{theorem}
If $A\in\mathcal{L}(H^n(G))$ where $H^n(G)$ is a discretizable Hilbert space, and if $A$ is closed, densely defined and dissipative then it generates a contractive semigroup.
\end{theorem}

If we represent parabolic equations by abstract evolution equations of the form
\begin{equation}
v'(t)=Av(t)+f(t)
\label{ev1}
\end{equation}
with $A\in\mathcal{L}(H^n(G))$, and where the Cauchy problem consists in finding a function $v\in C([0,T],H^n(G))\cap C^1([0,T),H^n(G))$ such that, for $t>0$, $v(t)\in dom(A)$ and \eqref{ev1} holds, and $v(0)=v_0$, where the initial value $v_0$ is prescribed.
It can be seen that the expression
\begin{equation}
e^{tA}:=\sum_{k=0}^\infty \frac{(tA)^k}{k!}
\end{equation}
satisfies the second and third contractive semigroup conditions, if we have that $A\in\mathcal{L}(H^n(G))$ is dissipative and $-A\in\mathcal{L}(H^n(G))$ is symmetric $V$-elliptic, then it can be verified that we will also have
\begin{equation}
\norm{}{e^{tA}v}\leq |\sum_{k\geq 0} \frac{(-t m)^k}{k!}|\norm{}{v} = e^{-mt}\norm{}{v}\leq \norm{}{v}, \:\: t\geq 0, v\in \text{dom}(A)
\end{equation}
with $m:=\inf_{x\neq 0}\norm{}{x}^{-2}|\mathcal{A}[x](x)|$, this clearly implies that $\{s_t:s_t:=e^{tA},t\geq0\}$ is a contractive semigroup. Using the last results we can obtain the following useful results.

\begin{corollary}\label{teu1}
If A is the generator of a contractive semigroup, then for each $u_0\in dom(A)$ there is a unique solution $u\in C^1([0,\infty),H^n(G))$ of \eqref{ev1} with $u(0)=u_0$.
\end{corollary}

\begin{theorem}\label{teu2}
If $A$ is the generator of a contractive semigroup, the for each $u_0\in dom(A)$ and each $f\in C^1([0,\infty),H^n(G))$ there is a unique $u\in C^1([0,\infty),H^n(G))$ such that $u(0)=u_0$, $u(t)\in dom(A)$ for $t\geq 0$, and
\begin{equation}
u'(t)=Au(t)+f(t), \:\:\: t\geq 0.
\label{eqt1}
\end{equation}
\end{theorem}
\begin{proof}
It suffices to show that the function
\begin{equation}
g(t)=\int_0^t s_{t-\tau}f(\tau)d\tau, \:\:\: t\geq 0,\nonumber
\end{equation}
satisfies \eqref{eqt1} and to note that $g(0)=0$. Letting $z=t-\tau$ we have
\begin{eqnarray*}
h^{-1}(g(t+h)-g(t))&=&\int_0^t s_z(f(t+h-z)-f(t-z))h^{-1}dz\\
&&+ h^{-1}\int_t^{t+h}s_z f(t+h-z)dz
\end{eqnarray*}
so it follows that $g'(t)$ exists and
\begin{equation}
g'(t)=\int_0^t s_z f'(t-z)dz + s_tf(0).\nonumber
\end{equation}
Furthermore we have
\begin{eqnarray}
h^{-1}(g(t+h)-g(t))&=&h^{-1}\left\{\int_0^t+hs_{t+h-\tau}f(\tau)d\tau-\int_0^ts_{t-\tau} f(\tau)d\tau\right\}\nonumber\\
&=&(s_h-\mathbf{1})h^{-1}\int_0^ts_{t-\tau}f(\tau)d\tau \nonumber\\
&&+h^{-1}\int_t^{t+h}s_{t+h-\tau}f(\tau)d\tau. \label{eq2t1}
\end{eqnarray}
Since $g'(t)$ exists and since the last term in \eqref{eq2t1} has a limit as $h\to0^+$, it follows from \eqref{eq2t1} that
\begin{equation}
\int_0^ts_{t-\tau}f(\tau)d\tau\in dom(A)\nonumber
\end{equation}
and that $g$ satisfies \eqref{eqt1}.
\end{proof}

\begin{example}
Given the initial boundary value problem described by
\begin{equation}
\left\{
\begin{array}{l}
\partial_x^2 u=\partial_tu, \:\: x\in (0,1) \\
u(0,t)=u(1,t)=0\\
u(x,0)=u_0(x), u_0\in H^1_0([0,1])
\end{array}
\right .
\label{ibv1}
\end{equation}
we can verify that
\begin{eqnarray}
\scalprod*{L^2([0,1])}{\partial_x^2 u}{v}&=&\int_{[0,1]} \partial_x^2 u \cdot\overline{v}\\
&=&-\int_{[0,1]} \partial_x u \cdot \partial_x \overline{v} \\
&=&-\scalprod*{H^1_0([0,1])}{u}{v}
\end{eqnarray}
wich clearly implies that $\text{Re}\mathcal{A}[x](x)\leq 0$ hence $A=\partial_x^2$ is dissipative and therefore $\{s_t:s_t=e^{t\partial_x^2}, t\geq 0\}$ is a contractive semigroup over $L^2([0,1])$. It follows from T.\ref{teu1} that \eqref{ibv1}
has a unique solution $u\in H^1_0(G\times[0,\infty))$.
\end{example}

\section{Discretization Schemes}
In this section we deal with the discretization processes applied to operators involved in partial differential equations wich permits us to obtain the discrete particular C*-algebras $\mathscr{A}_{n,h}(H^m(G))$, with $H^m(G)$ a discretizable Hilbert space, see \cite{Vides}, first we will present the following definitions.

\begin{definition}
  Grid: For two given sets $G \subseteq \mathbb{R}^N$ and $\mathbb{G}=\{0,
  \cdots, M_m \} \subseteq \mathbb{Z}$, with $M_m$ a number that depends on a
  fixed number $m$, a fixed value $h \in \mathbb{R}^N$ and a bijection $f :
  \mathbb{Z} \rightarrow \mathbb{R}^N : \mathbb{G} \ni k \mapsto g \in
  \mathbb{R}^N$, the set $G_{m, h} =\{g_k \in G : g_k = f (k), k \in
  \mathbb{G}\}$ is called a grid in $G$ of size $h$ and length $M_m$ or
  simply a grid.
\end{definition}

For a given discretizable Hilbert space $H^m(G)$ one can define an operator
$P_{n, h} \in \mathcal{L} (H^m(G))$ called particular projector and defined in the
following way.

\subsection{Particular Representation of Operators}

\begin{definition}
  Particular Projector: An operator $P_{n, h} \in \mathcal{L}(H^m(G))$, with 
  $\\H^m(G)$ a discretizable hilbert space, that satisfies the relations:
  \begin{eqnarray}
    P^2_{n, h} = P_{n, h} &  &  \label{proj1}\\
    P_{n, h} x \conv{}{h \to 0^+} x&  & \\
    \norm*{}{ P_{n, h} -\mathbf{1}} \leq c_{(\cdot)}
    h^{\mu_m} &  &  \label{proj2}
  \end{eqnarray}
  will be called a perticular projector, in (\ref{proj2}) $\norm*{}{\cdot}$ represents any prescribed norm in $\mathcal{L}(H)$ and
  $\mu_m$ is a number that depends on $m$ that will be called projection order
  with respect to $\norm*{}{\cdot}$.
\end{definition}

A particular projector $P_{n,h}\in\mathcal{L}(H^m(G))$ can be factored in the form $P_{n,h}=p_{n,h}p_{m,h}^\dagger$, where $p_{n,h}^\dagger\in\mathcal{L}(H^m(G),(H^m_{n,h}(G))^\ast)$ and $p_{n,h}\in\mathcal{L}((H^m_{n,h}(G))^\ast,\\ H^m_{n,h}(G))$ are called decomposition and expansion factors respectively. The expansion factor is related to a given grid of $G$ by the expression $p^\dagger u:=\{c_k(u,G_{n,h})\}=\hat{u}$ and the operator $P_{n,h}$ is related to a basis $\mathscr{P}:=\{p_k\}$ through $P_{n,h}p_k=p_k, \forall p_k\in\mathscr{P}$. Using the factors described above we can represent any particular projector over a discretizable Hilbert space $X(G)$ using the following diagram.

\begin{equation}
\xymatrix{
X(G) \ar[r]^{P_{m,h}} \ar[rd]^{p^\dagger_{m,h}} &
X_{m,h}(G)  \\
&X^\ast_{m,h}(G) \ar[u]_{p_{m,h}}}
\end{equation}

Using particular projectors we can define the particular representation of the inner product map in $H^m(G)$ and also of a given operator $B\in\mathcal{L}(H^n(G))$ as follows.

\begin{definition} Inner product matrix.
For a given discretizable Hilbert space $H$ whose inner product is induced by
the inner product map $\mathcal{M \in \mathcal{L}} (H, H^{\ast})$ in the following way
\begin{equation}
  \left\langle x, y \right\rangle_H := \mathcal{M} [x] (y)
\end{equation}
one can define a particular representation given by
\begin{equation}
  \mathcal{M}_{m, h} [p^{\dagger}_{m, h} \cdot] (p^{\dagger}_{m, h} \cdot)
  := \mathcal{M} [P_{m, h} \cdot] (P_{m, h} \cdot)
\end{equation}
that will recive the name of inner product matrix.
\end{definition}

A very important property of inner product matrices will be presented in the following.

\begin{theorem}
  Every inner product matrix form is symmetric and positive definite (SPD).
\end{theorem}  
  \begin{proof}
    It can be seen that for a discretizable Hilbert space $H$ and a given
    particular projector $P_{m, h}$ in $H$, with basis $\mathscr{P} =\{p_1,
    \cdots, p_{N_m} \}$, we will have that
    \begin{eqnarray*}
      \left(\mathcal{M}_{m,h}\right)_{i,j}&=& \mathcal{M}_{m,h}[p^{\dagger}_{m,h}p_i](p^{\dagger}_{m,h}p_j)\\
      &=& \left\langle P_{m, h} p_i, P_{m, h} p_j \right\rangle_H  \\
      &=& \left\langle p_i, p_j \right\rangle_H  \\
      &=& \overline{\left\langle p_j, p_i \right\rangle}_H  \\
      &=& \overline{\left\langle P_{m, h} p_j, P_{m, h} p_i \right\rangle}_H \\
      &=& \overline{\mathcal{M}_{m, h} [p^{\dagger}_{m, h} p_j] (p^{\dagger}_{m,
      h} p_i)} \\
      &=& \overline{(\mathcal{M}_{m, h})}_{j, i} 
    \end{eqnarray*}
    and this implies that $\mathcal{M}_{m, h} = \mathcal{M}^{\ast}_{m, h}$.
    Now since
    \[ 0 \leqslant \left\| x \right\|^2_{H_{m, h}} = \left\langle x, x
       \right\rangle_{H_{m, h}} =\mathcal{M}_{m, h} [p^{\dagger}_{m, h} x]
       (p^{\dagger}_{m, h} x) \]
    we will have that $\mathcal{M}_{m, h} [x] (x) > 0$ for each $0 \neq x \in
    H \backslash Ker \: P_{m, h}$.
  \end{proof}
  
This leads to the following
\begin{corollary}
Every inner product matrix form is invertible.
\end{corollary}

Properties of inner product matrices permit us to express the inner product and corresponding norm in $H_{n,h}$ for any $x,y\in H$ by
\begin{equation}
\scalprod*{H_{n,h}}{x}{y}=\scalprod*{2}{m_{n,h}x}{m_{n,h}y}=(m_{n,h}y)^\ast m_{n,h}x=y^\ast\mathcal{M}_{n,h}x
\end{equation}
and
\begin{equation}
\norm*{H_{n,h}}{x}=\scalprod*{H_{n,h}}{x}{x}^{1/2}=\scalprod*{2}{m_{n,h}x}{m_{n,h}y}=\norm*{2}{m_{m,h}x}
\end{equation}
respectively, with $m_{n,h}$ the formal square root of $mathcal{M}_{n,h}$.

\begin{definition}
  Particular Representation of an Operator. For a given operator $B \in
  \mathcal{L} (X, Y)$ being $X, Y$ discretizable Hilbert spaces and being
  $X_{m, h}, Y_{m, h}$ the subspaces relative to the particular projectors
  $P_{m, h} \in \mathcal{L} (X), Q_{m, h} \in \mathcal{L} (Y)$, the operator
  $B_{m, h} \in \mathcal{L} (X^{\ast}_{m, h}, Y^{\ast}_{m, h})$ given by
  \begin{equation}
    B_{m, h} := q^{\dagger}_{m, h} Bp_{m, h} \label{partop}
  \end{equation}
  will be called particular representation of $B$.
\end{definition}

Once we have computed the particular representation of a given operator $B \in
\mathcal{L} (X, Y)$ over a discretizable Hilbert spaces $X, Y$, in prescribed
subspaces $X_{m, h} \leqslant X, Y_{m, h} \leqslant Y$ determined by a
particular projectors $P_{m, h}, Q_{m, h}$, we will define the approximation
order of a particular representation as follows.

\begin{definition}
  Approximation order of a particular representation. We say that the
  particular representation $B_{m, h} \in \mathcal{L} (X^{\ast}_{m, h},
  Y^{\ast}_{m, h})$ of an operator $B \in \mathcal{L} (X, Y)$ is of order
  $\nu_m$ (with $\nu_m$ a value that depends of the prescribed number m) with
  respect to a given norm $\norm*{}{\cdot}$ in $Y$ if for each
  $x \in X$ there exists $c$ that does not depend on $h$ such that:
  \begin{equation}
    \norm*{}{ B_{m, h} p^{\dagger}_{m, h} x - q^{\dagger}_{m, h} Bx} \leq c_x h^{\nu_m}
  \end{equation}
\end{definition}

Based on the definition presented above and the properties of the inner product matrices if we denote by $\mathcal{M}_{m,h}(X)$ and $\mathcal{M}_{m,h}(Y)$ the particular representation of the inner product matrices of the discretizable Hilbert spaces $X$ and $Y$ respectively we can obtain an important definition for $\norm*{\mathcal{L}(X_{m,h},Y_{m,h})}{\cdot}$ in the following way.

\begin{eqnarray}
\norm*{\mathcal{L}(X_{m,h},Y_{m,h})}{A}&:=&\sup_{v\neq 0} \frac{\norm*{Y_{m,h}}{Av}}{\norm*{X_{m,h}}{v}}\\
&=&\sup_{v\neq 0} \frac{\norm*{2}{m_{m,h}(Y)Av}}{\norm*{2}{m_{m,h}(X)v}}\\
&=&\sup_{z=m_{m,h}(X)v\neq 0} \frac{\norm*{2}{m_{m,h}(Y)Am_{m,h}(X)^{-1}z}}{\norm*{2}{z}}\\
&=&\norm*{2}{m_{m,h}(Y)Am_{m,h}(X)^{-1}}
\end{eqnarray}
wich permits us to mimic important properties of operator algebras in $\mathcal{L}(X)$ wich is really important when we deal wich C*-algebras of operators in discretizable Hilbert spaces. When there is no confussion we will just refer to the norm in the corresponding Banach space (discrete or not) simply by $\norm{}{\cdot}$. As a consequence of the Gershgorin theorem it is not difficult to derive the following.

\begin{corollary}\label{gersh}
$\norm{}{A}\leq\norm*{\infty}{m_{n,h}Am_{n,h}^{-1}}, \forall A\in\mathscr{A}_{n,h}(G)$.
\end{corollary}

\begin{definition}
Discrete Particular C*-algebra of Operators. If a particular C*-algebra $\mathscr{A}$ is generated by finitely many operators say $\{\mathbf{1},a,b,\cdots\}\subset\mathcal{L}(H^n(G))$, then the C*-algebra $\mathscr{A}_{n,h}$ defined by $\mathscr{A}_{n,h}:=\{\hat{b}\in\mathcal{L}(H_{n,h}^m(G)): \hat{b}:=p^\dagger_{n,h}bp_{n,h}, b\in\mathscr{A}(H^m(G))\}$ will be called discrete particular C*-algebra or simply particular C*-algebra generated by $\{\mathbf{1},a_{n,h},b_{m,h},\cdots\}$.
\end{definition}

\subsection{Spectral Estimates}

We will begin this section defining a very useful condition described as follows.
\begin{definition} Stability. A sequence $\{A_{n,h}\}_{0<h\leq h_0}$, $n\geq n_0$ of $m\times m$ matrices $A_{n,h}$ is said to be $h$-stable if the matrices $A_{n,h}$ are invertible for any fixed $n\geq n_0\in \mathbb{Z}^+$ and all sufficiently small $h$, say $0<h\leq h_1$ for any fixed $n\in\mathbb{N}$, and if
\begin{equation}
\sup_{\mathbb{R}^+\ni h\leq h_0}\norm{}{A_{n,h}^{-1}}<\infty.
\end{equation}
\end{definition}

If $B$ is not invertible we put $\norm{}{B^{-1}}=\infty$. With this convention, we can say that the sequence $\{A_{n,h}\}_{n\geq n_0,h\leq h_0}$ is $h$-stable if and only if
\begin{equation}
\limsup_{h\to 0^+} \norm{}{A_{n,h}^{-1}}<\infty, \text{  for fixed } n\in\mathbb{N}.
\end{equation}

Using the expressiones presented above we can obtain the following result.

\begin{theorem}\label{stability1}
If the sequence $\{A_{n,h}\}_{n\geq n_0,0<h\leq h_0}$ of $m\times m$ matrices $A_{n,h}$ is stable, then $A$ is necessarily invertible.
\end{theorem}
\begin{proof}
Let $\norm{}{A_{n,h}^{-1}}\leq M$ for $0<h\leq h_0$ and $n\in\mathbb{N}$ fixed. Then if $x\in H$ with $H$ a discretizable Hilbert space and $0<h\leq h_0$,
\begin{eqnarray*}
\norm{}{P_{n,h}x}&=&\norm{}{A_{n,h}^{-1}A_{n,h}x}\leq M\norm{}{A_{n,h}x}=M\norm{}{P_{n,h}AP_{n,h}x},\\
\norm{}{P_{n,h}x}&=&\norm{}{(A^\ast_{n,h})^{-1}A^\ast_{n,h}x}\leq M\norm{}{A^\ast_{n,h}x}=M\norm{}{P_{n,h}A^\ast P_{n,h}x},
\end{eqnarray*}
and passing to the limit $h\to 0^+$, we get
\begin{equation}
\norm{}{x}\leq M\norm{}{Ax}, \:\:\: \norm{}{x}\leq M\norm{}{A^\ast x}
\end{equation}
for every $x\in H$. This shows that $A$ is invertible.
\end{proof}

Now we will present a useful definition and some important related results and diagrams.

\begin{definition}Exactly Factorizable Operator. 
An operator $A\in \mathscr{A}(X(G))$ with $\mathscr{A}(X(G))$ a C*-algebra of operators over a discretizable Hilbert space 
$X(G)$, that can be factored in the form $A=a^\dagger a$, with $a\in \mathscr{A}(X(G))$ is said to be exactly factorizable in $\mathscr{A}(X(G))$ and will satisfy the following estimates.
\end{definition}

Exactly factorizable operators can be represented by the following diagram.
\begin{equation}
\xymatrix{
X(G) \ar[r]^A \ar[rd]^{a^\dagger} &
X''(G)  \\
&X'(G) \ar[u]^a}
\label{diag2}
\end{equation}
If $X(G)$ is a discretizable Hilbert space, a particular representation $a_{m,h}\in X^\ast_{m,h}(G)$  of $a\in\mathscr{A}(X(G))$ with respect to a particular projection $P_{m,h}\in\mathcal{P}(X(G))$, being $\mathcal{P}(X(G))$ the space of all projections over $X(G)$, can be expressed using the following diagram.
\begin{equation}
\xymatrix
{   X(G) \ar[r]^{p^\dagger_{m,h}} \ar[rd]^{P_{m,h}} & X^\ast_{m,h}(G) \ar[d]^{p_{m,h}} \ar[r]^{a_{m,h}} & X'^\ast_{m,h}(G) \ar[r]^{p_{m,h}} & X'_{m,h}(G) \\
         & X_{m,h}(G) \ar[r]^{a} & X'_{m,h}(G) \ar[u]_{p_{m,h}^\dagger} \ar[ur]_{P_{m,h}}
}
\label{diag3}
\end{equation}

Using diagrams \eqref{diag2} and \eqref{diag3} we can express the particular factorization of an exactly factorizable operator $A\in\mathscr{A}(X(G))$ by the following diagram.
\begin{equation}
\xymatrix{
{X'}_{m,h}^\ast(G) \ar[d]_{\mathcal{M}_{m,h}}  &X_{m,h}^\ast(G)\ar[l]_{a_{m,h}^\ast}\ar[d]_{\mathcal{A}_{m,h}}\ar[dr]^{A_{m,h}}\ar[r]^{a_{m,h}^\dagger} & {X'}^\ast_{m,h}(G) \ar[d]^{a_{m,h}}\\
{X'}_{m,h}^\ast(G) \ar[r]_{a_{m,h}} &{X''}^\ast_{m,h}(G) \ar[r]_{\mathcal{M}_{m,h}^{-1}} &{X''}^\ast_{m,h}(G)}
\label{diag4}
\end{equation}
Now we will present the following results applicable to exactly factorizable operators.

\begin{lemma}
Let $A\in \mathscr{A}(X(G))$ be exactly factorizable in $\mathscr{A}(X(G))$. Then
\begin{equation}
|\mathcal{A}[x](x)|\geq \varepsilon \norm{}{x}^2.
\end{equation}
\end{lemma}
\begin{proof}
It can be seen that
\begin{eqnarray}
|\mathcal{A}[x](x)|&=&|\scalprod{}{Ax}{x}|\\
&=&|\scalprod{}{a^\dagger ax}{x}|\\
&=&|\scalprod{}{ax}{ax}|\\
&=&\norm{}{ax}^2
\end{eqnarray}
wich implies that $|\mathcal{A}[x](x)|$ is (SPD) and that there exists $\delta>0$ such that
\begin{equation}
\norm{}{ax}\geq \delta \norm{}{x}.
\end{equation}
taking $\varepsilon=\delta^2$ concludes the proof.
\end{proof}

\begin{theorem}\label{tspec1}
Let $A\in \mathscr{A}(X(G))$ be exactly factorizable in $\mathscr{A}(X(G))$. Then we will have
\begin{eqnarray}
\norm{}{A_{n,h}}\leq \norm{}{A} , n\geq 1, \:\:\:\: \lim_{h\to 0^+}\norm{}{A_{n,h}}=\norm{}{A}\nonumber\\
\norm{}{A_{n,h}^{-1}}\leq \norm{}{A^{-1}} , n\geq 1, \:\:\:\:\:\: \lim_{h\to 0^+}\norm{}{A_{n,h}^{-1}}=\norm{}{A^{-1}}\nonumber.
\end{eqnarray}
\end{theorem}
\begin{proof}
Lets define
\begin{eqnarray}
m&:=&\inf_{x\neq 0}\frac{\scalprod{}{Ax}{x}}{\scalprod{}{x}{x}}, M:=\sup_{x\neq 0}\frac{\scalprod{}{Ax}{x}}{\scalprod{}{x}{x}}\\
m_{n,h}&:=&\inf_{x\neq 0}\frac{\scalprod*{n,h}{AP_{n,h}x}{x}}{\scalprod*{n,h}{x}{x}}, M_{n,h}:=\sup_{x\neq 0}\frac{\scalprod*{n,h}{AP_{n,h}x}{x}}{\scalprod*{n,h}{x}{x}}
\end{eqnarray}
By assumption, $m\geq \varepsilon >0$. We have
\begin{eqnarray}
m_{n,h}&=&\inf_{x\neq 0}\frac{\scalprod*{m,h}{AP_{n,h}x}{x}}{\scalprod*{n,h}{x}{x}}=\inf_{x\neq 0}\frac{\scalprod{}{P_{n,h}AP_{n,h}x}{P_{n,h}x}}{\scalprod{}{P_{n,h}x}{P_{n,h}x}}\\
&=&\inf_{x\neq 0}\frac{\scalprod{}{AP_{n,h}x}{P_{n,h}x}}{\scalprod{}{P_{n,h}x}{P_{n,h}x}}\geq \inf_{x\neq 0}\frac{\scalprod{}{Ax}{x}}{\scalprod{}{x}{x}}=m
\end{eqnarray}
and , analogously,
\begin{eqnarray}
M_{n,h}&=&\sup_{x\neq 0}\frac{\scalprod*{m,h}{AP_{n,h}x}{x}}{\scalprod*{n,h}{x}{x}}=\sup_{x\neq 0}\frac{\scalprod{}{P_{n,h}AP_{n,h}x}{P_{n,h}x}}{\scalprod{}{P_{n,h}x}{P_{n,h}x}}\\
&=&\sup_{x\neq 0}\frac{\scalprod{}{AP_{n,h}x}{P_{n,h}x}}{\scalprod{}{P_{n,h}x}{P_{n,h}x}}\leq \sup_{x\neq 0}\frac{\scalprod{}{Ax}{x}}{\scalprod{}{x}{x}}=M.
\end{eqnarray}
Because
\begin{equation}
\norm{}{A}=M, \:\: \norm{}{A_{n,h}}=M_{n,h}, \:\: \norm{}{A^{-1}}=1/m, \:\: \norm{}{A_{n,h}^{-1}}=1/m_{n,h},
\end{equation}
we arrive at the inequalities $\norm{}{A_{n,h}}\leq\norm{}{A}$ and $\norm{}{A_{n,h}^{-1}}\leq\norm{}{A^{-1}}$ for all $n$ and $0<h<1$. It is clear that $\norm{}{A_{n,h}}\to \norm{}{A}$, and since
\begin{equation}
\norm{}{A^{-1}}\leq\liminf_{h\to0^+}\norm{}{A_{n,h}^{-1}}\leq\limsup_{h\to0^+}\norm{}{A_{n,h}^{-1}}\leq\norm{}{A^{-1}},
\end{equation}
it follows that $\norm{}{A_{n,h}^{-1}}\conv{}{h\to0^+}\norm{}{A^{-1}}$, for each $n\geq 1$.
\end{proof}

If we denote by $\lambda_{min}(A)$ and $\lambda_{max}(A)$ the minimal and maximal eigenvalues of an operator $A\in\mathscr{A}(X)$ respectively, with $\mathscr{A}(X)$ a C*-algebra of operators over a Banach space $X$, an operator $A\in\mathscr{A}(H^n(G))$ that is exactly factorizable in $\mathscr{A}(H^n(G))$ satifies the following estimate.

\begin{theorem}
If $A\in\mathscr{A}(H^n(G))$ is exactly factorizable in $\mathscr{A}(H^n(G))$, then
\begin{eqnarray}
&&m\leq\lambda_{min}(A_{n,h})\leq\lambda_{max}(A_{n,h})\leq M,\label{spec1}\\
&&\lim_{h\to0^+}\lambda_{min}(A_{n,h})=m, \lim_{h\to0^+}\lambda_{max}(A_{n,h})=M\label{spec2}
\end{eqnarray}
\begin{equation}
\{m,M\}\subset\sigma(A)\subset\liminf_{h\to0^+}\sigma(A_{n,h})\subset\limsup_{h\to0^+}\sigma(A_{n,h})\subset[m,M]\label{spec3}.
\end{equation}
\end{theorem}
\begin{proof}
The validity of \eqref{spec1} and \eqref{spec2} was established in the proof of T.\ref{tspec1}. The only nontrivial part of \eqref{spec3} is the inclusion
\begin{equation}
\sigma(A)\subset\liminf_{h\to0^+}\sigma(A_{n,h}) \label{nontrivial}.
\end{equation}
So assume $\lambda\in\mathbb{R}$ is not in $\liminf \sigma(A_{n,h})$. Then there is an $\varepsilon>0$ such that
\begin{equation}
U_\varepsilon(\lambda)\cap\sigma(A_{n,h})=\emptyset \: \: \text{for all} \: \: n\geq n_0,
\end{equation}
where $U_\varepsilon(\lambda):=\{z\in\mathbb{C}:|z-\lambda|<\varepsilon\}$. Hence $U_\varepsilon(0)\cap\sigma(A_{n,h}-\lambda\mathbf{1})$ for all $n\geq n_0$, and since $(A-\lambda\mathbf{1})^{-1}$ is symmetric, and therefore the norm coincides with the spectral radius, we get
\begin{equation}
\norm{}{(A-\lambda\mathbf{1})^{-1}}\leq 1/\varepsilon \: \: \text{for all} \: \: n\geq n_0.
\end{equation}
It follows that $\{A_{n,h}-\lambda \mathbf{1}\}_{n\geq n_0}$ is stable, and therefore $A-\lambda\mathbf{1}$ must be invertible by theorem T.\ref{stability1}. Consequently, $\lambda \notin \sigma(A)$, wich completes the proof of \eqref{nontrivial}.
\end{proof}

If an exactly factorizable operator $A\in\mathscr{A}(X(G))$ is $X(G)$-coercive and $X(G)$ is discretizable we can verify that
\begin{eqnarray}
|\mathcal{A}_{m,h}[\hat{x}](\hat{x})|&=&\scalprod{}{P_{m,h}AP_{m,h}x}{P_{m,h}x}=\scalprod{}{a^\dagger aP_{m,h}x}{P_{m,h}x}\\
&=&\scalprod{}{aP_{m,h}x}{aP_{m,h}x}=\scalprod{}{P_{m,h}aP_{m,h}x}{P_{m,h}aP_{m,h}x}\\
&=&\scalprod*{X_{m,h}(G)}{a_{m,h}\hat{x}}{a_{m,h}\hat{x}}\geq m \norm{}{x}^2
\label{eqp1}
\end{eqnarray}
this leads us to the following.

\begin{theorem}
If an exactly factorizable operator $A\in\mathscr{A}(X(G))$ is $X(G)$-coercive and $X(G)$ is discretizable then $A_{m,h}:=p^\dagger_{m,h}Ap_{m,h}$ is invertible.
\end{theorem}
\begin{proof}
From \eqref{eqp1} it is clear that $A_{m,h}$ is positive definite hence invertible.
\end{proof}

\subsection{Discrete Time Integration} In general, evolution equations of the types studied here can be expressed in the form
\begin{equation}
\left \{
\begin{array}{l}
u'(t)=f(u(t))\\
u(0)=u_0
\end{array}
\right .
\label{gform}
\end{equation}
if we define $\mathcal{H}^k_\delta(G;r)$ by
\begin{equation}
\mathcal{H}^k_\delta(G;r):=\{v\in C([-\delta,\delta],H^1(G)):v\in\overline{B}_r(0), \forall t\in [-\delta,\delta]\}
\end{equation}
and taking the operator $T:\mathcal{H}^k_\delta(G;r)\rightarrow H^1(G_\tau)$ to be defined by
\begin{eqnarray}
T&:&\mathcal{H}^k_\delta(G;r)\longrightarrow H^1(G_\tau)\\
&:&v\longmapsto v_0+\int_0^tf(v(\tau))d\tau
\label{opT}
\end{eqnarray}
and if $f\in C^{\alpha=1}(\mathcal{H}^k(G;r))$ then it can be seen that if we define $\norm*{\mathcal{H}^k_\delta(G;r)}{\cdot}$ by
\begin{equation}
\norm*{\mathcal{H}^k_\delta(G;r)}{v}:=\sup_{t\in[-\delta,\delta]}\norm{}{v}
\end{equation}
where like in the above sections $\norm{}{\cdot}:=\norm*{H^n(G)}{\cdot}$, from these expressions we get
\begin{equation}
\norm{}{T(v)-v_0}=\norm{}{\int_0^tf(v(\tau))d\tau}\leq M_f\delta \label{proof1}
\end{equation}
where $M_f:=\sup_{v\in\mathcal{H}^k(G;r)}\norm{}{f(v)}$, so if $M_f\delta < R$, then $T(v)\in\mathcal{H}^k_\delta(G;r)$ so that $T\in\mathcal{L}(\mathcal{H}^k_\delta(G;r))$. Also we have
\begin{eqnarray}
\norm*{\mathcal{H}^k_\delta(G;r)}{T(u)-T(v)}&=&\sup_{t\in[-\delta,\delta]}\norm{}{\int_0^t(f(u(\tau))-f(v(\tau)))d\tau}\label{proof2}\\
&\leq&c_f\norm*{\mathcal{H}^k_\delta(G;r)}{u-v}\int_0^\delta d\tau\\
&\leq&c_f\delta\norm*{\mathcal{H}^k_\delta(G;r)}{u-v} \label{proof3}
\end{eqnarray}
so if we choose $\delta<\min \{R/M_f,1/c_f\}$ we can see that $T\in\mathcal{L}(\mathcal{H}^k_\delta(G;r))$ will be a strict contraction. The above results con be sumarized in the following.

\begin{theorem} \label{picard} Picard Existence theorem. Let $Y(G)$ be a Banach space. Suppose that $f\in C^{\alpha=1}(\mathcal{Y}_\delta(G;r)), \mathcal{Y}_\delta(G;r)\subseteq Y(G)$. Let $M_f:=\sup_{v\in\mathcal{Y}_\delta(G;r)}\norm{}{f(v)}$. The initial value problem \eqref{gform}
has a unique local solution u(t). This solution is defined for $t\in(-\delta,\delta)$ with $\delta=\min\{R/M_f,1/c_f\}$.
\end{theorem}
\begin{proof}
From \eqref{proof1} and from \eqref{proof2} to \eqref{proof3} we have that $T\in\mathcal{\mathcal{Y}_\delta(G;r)}$ is a strict contraction for $\delta<R/M_f$, from this fact the result follows.
\end{proof}

\begin{corollary}\label{corollary2}
The solution $u(t)$ to \eqref{gform} satisfies the estimate
\begin{equation}
\norm*{\mathcal{H}^k_\delta(G;r)}{u(t)-u_k(t)}\leq\frac{(c_f\delta)^k}{1-(c_f\delta)}M_f\delta.
\end{equation}
\end{corollary}
\begin{proof}
Follows from remark R.\ref{rcontraction}.
\end{proof}

Time integration of space discretized evolution equations of the form 
\begin{equation}
\left\{
\begin{array}{l}
x'(t)=F(x(t)) \\
x(0)=x_0
\end{array}
\right .
\label{ev01}
\end{equation}
where $F \in C^{\alpha=1}(\Omega), \Omega\subset\mathbb{C}^N$ can be performed using the operator $S:\mathbb{C}^N\times \mathbb{R}\longrightarrow \mathbb{C}^N$ defined by
\begin{equation}
S[y_k(t_0)](t):=y_k(t_0)+\int_{t_0}^tF(y_k(s))ds
\label{ev03}
\end{equation}
with $y_{k}(t):=T^k[x_0](t)$ and where $T\in\mathcal{L}(\mathbb{C}\times\mathbb{R},\mathbb{C})$ is a succesive approximation operator analogous to the one defined in \eqref{opT}, it can be seen that \eqref{ev03} can be expressed in the form
\begin{equation}
S[T^k[x_0](t_0)](t)=T^k[x_0](t_0)+\int_{t_0}^tF(T^k[x_0](s))ds
\label{ev04}
\end{equation}
from this equation it can be obtained that $S[T^k[x_0](t_0)](t_0)=T^k[x_0](t_0)$ wich implies that $S[(\cdot)(t_0)](t_0)=\mathbf{1}$, it can be seen also that
\begin{eqnarray*}
S[T^k[x_0](t)](t')&=&S[S[T^k[x_0](t_0)](t)](t') \\
&=&T^k[x_0](t_0)+\int_{t_0}^tF(T^{k}[x_0](s))ds +\int_{t}^{t'} F(T^k[x_0](s))ds\\
&=&T^k[x_0](t_0)+\int_{t_0}^{t'} F(T^k[x_0](s))ds\\
&=&S[T^k[x_0](t_0)](t')
\end{eqnarray*}
Now if we take a time interval $[t_0,\tau]$ to be partitioned in the form $t_n:=t_0+nh, n=0,\cdots,N$ and where $h=(\tau-t_0)/N$ the last relation presented above permits us to obtain a discrete expression for the time integration operator $S_h[x_n](m)$ in the following way
\begin{equation}
S_{n}[x_k(0)]:=S[T^k[x_0](t_0)](n)=T^k[x_0](t_0)+\int_{t_0}^{t_0+nh}F(T^k[x_0](s))ds
\label{dsem1}
\end{equation}
using this expression and without any lack of generality it can be observed that for any two $m,n\in\mathbb{Z}_0^+$ such that $m>n$ we will have
\begin{eqnarray}
S_m\circ S_n[x_k(0)]&=&S_n[x_k(0)]+\int_{t_0+nh}^{t_0+nh+mh}F(T^k[x_0](s))ds\\ \nonumber
&=&T^k[x_0](t_0)+\int_{t_0}^{t_0+nh}F(T^k[x_0](s))ds\\
&&+\int_{t_0+nh}^{t_0+nh+mh}F(T^k[x_0](s))ds\\
&=&T^k[x_0](t_0)+\int_{t_0}^{t_0+nh+mh}F(T^k[x_0](s))ds\\
&=&S_{m+n}[x_k(0)] \label{dsem2}
\end{eqnarray}
From \eqref{dsem1} and from \eqref{dsem2} it can be seen that $S_0[x]=x$ and also that the set $\{S_{n}[T^k[\cdot]] : n \in \mathbb{Z}_0^+\}$ defines a discrete semigroup with respect to composition. The action of discrete time integration operators over fixed elements of the approxmate solution to the evolution equations can be computed using quadrature rules with accuracy order higher or equal to the one induced by succesive approximation operators.

\section{Approximation of Nonlinear Evolution Equations}

The nonlinear evolution equations studied in this work will be related in some suitable sense to elliptic operators not necesarily linear.

\subsection{Approximation of Semilinear Elliptic Problems}
For a given semilinear problem of the form:
\begin{equation}
\left\{
\begin{array}{l}
A(u)=f(u), x\in G \\
Bu=u_b, x\in \partial G
\end{array}
\right .
\label{eq2}
\end{equation}
where $A(\cdot)\in\mathcal{L}(H^n(G))$ is $H^n(G)$-coercive exactly factorizable on $\mathscr{A}(H^n(G))$, one can obtain a particular representation of $A$ in $\mathscr{A}_{n,h}(H^m(G))$ using a particular projector compatible in some sense with the boundary conditions of $u\in H^m(G)$ in \eqref{eq2}. Now, if we have that  $f$ is locally H{\"o}lder continuous with exponent one ,i.e., $f\in C^\alpha(\overline{B}_r(0))$, with $\overline{B}_r(0)$ a closed ball of radius $r$ centered in $0$, from here on we will define the set $\mathcal{H}^n(G;r)\subseteq H^n(G)$ by
\begin{equation}
\mathcal{H}^m(G;r):=\{u\in H^m(G):u\in\overline{B}_r(0), \forall x\in \overline{G}\}
\end{equation}
also if there is an $0<\varepsilon_r\in\mathbb{R}$ that depends on $r$, small enough, such that there exists $c_f<\infty$ with
\begin{equation}
\norm*{H^m(G)}{f(u)-f(v)}\leq c_f\norm*{H^m(G)}{u-v}
\end{equation}
when $\norm*{H^m(G)}{u-v}\leq\varepsilon_r$, then we can obtain the following.

\begin{theorem}\label{tsem1}
If for the semilinear boundary value problem
\begin{equation}
\left\{
\begin{array}{l}
A(u)=f(u), x\in G \\
Bu=u_b, x\in \partial G
\end{array}
\right .
\label{eq3}
\end{equation}
we have that $f\in C^\alpha(\overline{B}_r(0))$ with $\alpha=1$, then we will have that the sequence defined by $\hat{u}_0:=p_{m,h}^\dagger u_0, \hat{u}_{k+1}:=\mathcal{G}_{n,h}f(\hat{u}_k)$ for a given $u_0\in X(G)$ converges to a unique approximation of the solution to \eqref{eq3} determined by $u_0$.
\end{theorem}
\begin{proof}
It suffices to show that under some particular conditions $\mathcal{G}_{m,h}f(\cdot)$ is a contraction, it can be seen that for any $u,v\in X(G)$ such that $Rg(u),Rg(v)\subseteq\overline{B}_r(u_0)$ for any $x\in\overline{G}$ we have
\begin{eqnarray}
\norm{}{\mathcal{G}_{m,h}f(u)-\mathcal{G}_{m,h}f(v)}&\leq&\norm{}{\mathcal{G}_{m,h}}\norm{}{f(u)-f(v)}\\
&\leq&\norm{}{\mathcal{G}_{m,h}}c_f(r)\norm{}{u-v}\\
&\leq&\frac{c_f(r)}{m}\norm{}{u-v}
\end{eqnarray}
the last expression implies that $r$ must be chosen such that $c_f(r)<m$ wich implies that $\mathcal{G}_{m,h}f(\cdot)$ becomes a contraction, and by T.\ref{contraction} the result follows.
\end{proof}

\begin{theorem}\label{tsem2}
If $u_h\in X_{m,h}(G)$ denotes the numerical solution to the particular representation of \eqref{eq3} given by
\begin{equation}
\left\{
\begin{array}{l}
A_{m,h}(u_h)=f(u_h), x\in G \\
B_{m,h}u_h=p^\dagger_{m,h}u_b, x\in \partial G
\end{array}
\right .
\end{equation}
and if $u:=\lim_{h\to0^+}u_h$, then we will have that
\begin{equation}
\norm{}{u-\hat{u}_k}\leq c_uh^{\nu_m}+\frac{(c_f(r)/m)^k}{m-c_f} (r\norm{}{A_{m,h}}+M_f)\mu(G)
\end{equation}
where $M_f:=\sup_{u\in \mathcal{H}^n(G;r)}|f(u)|$, $\mu(G)=\norm{}{1}$ and $\nu_m$ is the approximation order of the particular projection.
\end{theorem}
\begin{proof}
From R.\ref{rcontraction} we have
\begin{eqnarray}
\norm{}{u_h-\hat{u}_k}&\leq&\frac{(c_f(r)/m)^k}{1-(c_f(r)/m)}\norm{}{\hat{u}_1-\hat{u}_0}\\
&\leq&\frac{(c_f(r)/m)^k}{1-(c_f(r)/m)}\norm{}{\mathcal{G}_{m,h}f(\hat{u}_0)-\hat{u}_0}\\
&\leq&\frac{(c_f(r)/m)^k}{1-(c_f(r)/m)}\norm{}{\mathcal{G}_{m,h}}\norm{}{A_{m,h}(\hat{u}_0)-f(\hat{u}_0)}\\
&\leq&\frac{(c_f(r)/m)^k}{(m-c_f(r))}(\norm{}{A_{m,h}}\norm{}{\hat{u}_0}+\norm{}{f(\hat{u}_0)})\\
&\leq&\frac{(c_f(r)/m)^k}{(m-c_f(r))}(\norm{}{A_{m,h}}r\norm{}{1}+M_f\norm{}{1})\\
&\leq&\frac{(c_f(r)/m)^k}{(m-c_f(r))}(r\norm{}{A_{m,h}}+M_f)\mu(G)
\end{eqnarray}
also it can be seen that
\begin{eqnarray}
\norm{}{u-\hat{u}_k}&\leq&\norm{}{u-u_h}+\norm{}{u_h-\hat{u}_k}\\
&\leq&\norm{}{u-P_{m,h}u}+\frac{(c_f(r)/m)^k}{(m-c_f(r))}(r\norm{}{A_{m,h}}+M_f)\mu(G)\\
&\leq&c_uh^{\nu_m}+\frac{(c_f(r)/m)^k}{(m-c_f(r))}(r\norm{}{A_{m,h}}+M_f)\mu(G)
\end{eqnarray}
from the last expression we get the desired result.
\end{proof}

\begin{example}
Given the semilinear boundary value problem:
\begin{equation}
\left\{
\begin{array}{l}
\Delta u=1+u^2, x\in [0,1]^2 \\
u=0, x\in \partial [0,1]^2
\label{eqex1}
\end{array}
\right .
\end{equation}
with $\Delta:=\partial_x^2+\partial_y^2$, one can rewrite it in a particular representation of the form
\begin{equation}
(a^\ast_{2,1/32}\mathcal{M}_{2,1/32}a_{2,1/32})u_h=\mathcal{M}_{2,1/32}(1+u_h^2)
\end{equation}
where $a_{2,1/32}\in\mathscr{A}_{2,1/32}(H^1_0(G))$ is the particular representation of the exact factor of $A_{2,1/32}:=p^\dagger_{2,1/32}\Delta p_{2,1/32}$ and $\mathcal{M}_{2,1/32}$ is the inner product matrix form relative to $H^1_0(G)$. Taking $\mathcal{G}_{2,1/32}$ to be defined by
\begin{equation}
\mathcal{G}_{2,1/32}:=(a^\ast_{2,1/32}\mathcal{M}_{2,1/32}a_{2,1/32})\mathcal{M}_{2,1/32}
\end{equation}
and since a simple computation permits us to see that $m=2\pi^2$, we have $\norm{}{\mathcal{G}_{2,1/32}}\leq(2\pi^2)^{-1}$, on the other hand it can be seen that
\begin{eqnarray}
\norm*{L^2([0,1]^2)}{f(u)-f(v)}&=&\norm*{L^2([0,1]^2)}{u^2-v^2}\\
&\leq&(\norm*{L^2([0,1]^2)}{u}+\norm*{L^2([0,1]^2)}{v})\\ \nonumber
&&\times \norm*{L^2([0,1]^2)}{u-v}\\
&\leq&2r\norm*{L^2([0,1]^2)}{u-v}
\end{eqnarray}
now combining the results presented above we obtain
\begin{eqnarray}
\norm*{L^2([0,1]^2)}{\mathcal{G}_{2,1/32}f(u_{1/32})-\mathcal{G}_{2,1/32}f(v_{1/32})}&\leq&\norm{}{\mathcal{G}_{2,1/32}}\nonumber\\ &&\times\norm*{L^2([0,1]^2)}{f(u_{1/32})-f(v_{1/32})}\nonumber\\
&\leq&\frac{r}{\pi^2}\norm*{L^2([0,1]^2)}{u_{1/32}-v_{1/32}}
\end{eqnarray}
so we need to choose $r<\pi^2$ to ensure that $\mathcal{G}_{m,h}f(\cdot)$ becomes a strict contraction, then by theorems T.\ref{tsem1} and T.\ref{tsem2} we will have that \eqref{eqex1} has a unique approximate solution $\hat{u}_k$ that satisfies the following estimate
\begin{equation}
\norm{}{u-\hat{u}_k}\leq c_u(1/32)^{4}+\frac{(r/\pi^2)^k}{2\pi^2-2r} (rK_A(32)^2+1+r^2).
\end{equation}
\end{example}

\subsection{Approximation of Dissipative Nonlinear Evolution Equations}
In this section we will focus our attention in evolution equations of the form
\begin{equation}
\left\{
\begin{array}{l}
A(u(t))+f(u(t))=u'(t) \\
u(0)=u_0, u_0\in H^m(G)
\end{array}
\right .
\label{eveq2}
\end{equation}
where such that $A(\cdot)\in\mathscr{A}(H^m(G))$ will be an exactly factorizable operator on the particular C*-algebra $\mathscr{A}(H^n(G))$ of the form
\begin{equation}
A(v):=\left\{
\begin{array}{l}
[a^\dagger\mathcal{D}(v)a] v) , x\in G\\
Bv=v_b, x\in\partial G
\end{array}
\right .
\end{equation}
with $D(v)\in\mathbb{C}^{n\times n}$ a symmetric Tensor function, i.e. $D(v)^\ast=D(v)$ and where $f\in C^{\alpha=1}(\mathcal{H}^k(G;r))$. 

In the following sections we will derive several useful estimates.

\subsubsection{Diffusion Tensor $\mathcal{D}(u):=\mathcal{K}$}\label{sect1}

Using the above results, the first case that we will study is the one consisting of $D(v):=\mathcal{K}$, with $\mathcal{K}$ constant in time and (SPD), in this cases and all the other we will study in this work we will consider that the boundary operator $B\in\mathcal{L}(H^n(G))$ allows $A(\cdot)\in\mathscr{A}(H^n(G))$ to be exactly factorizable in $\mathscr{A}(H^n(G))$ and that $f\in C^{\alpha=1}(\mathcal{H}^n(G;r))$.

If we take $g\in\mathcal{F}(H^1(G_\tau))$ to be the right hand side of \eqref{eveq2}, i.e. $g(v):=A(v)+f(v)$ and if we denote by $A[\mathcal{I}]$ the form corresponding to $A(\cdot)$ when $D(u):=\mathcal{I}=\mathbf{1}$ then for any two $u,v\in dom(A(\cdot))\cap \mathcal{H}^k_\delta(G;r)$ and since $f\in C^{\alpha=1}(\mathcal{H}^n(G;r))$ we will have
\begin{eqnarray}
\norm{}{g(u)-g(v)}&=&\norm{}{A(u-v)+f(u)-f(v)}\\
&\leq&\norm{}{A(u-v)}+\norm{}{f(u)-f(v)}\\
&\leq&\norm{}{a^\dagger\mathcal{K}a}\norm{}{u-v}+\norm{}{f(u)-f(v)}\\
&\leq&\norm{}{a}^2\norm{}{\mathcal{K}}\norm{}{u-v}\norm{}{f(u)-f(v)}\\
&\leq&\norm{}{\mathcal{K}}\norm{}{A[\mathcal{I}]}\norm{}{u-v}+c_f\norm{}{u-v}\\
&=&\norm{}{\mathcal{K}}\norm{}{A[\mathcal{I}]}+c_f)\norm{}{u-v}
\end{eqnarray}
hence $g\in C^{\alpha=1}(\mathcal{H}^n(G;r))$. In the following section we will generalize these results and present a useful technical theorem and its corresponding corollary.

\subsubsection{Diffusion Tensor $\mathcal{D}(\cdot)\in C^{\alpha=1}_+(\mathcal{H}^k_\delta(G;r),\mathbb{C}^{n\times n})$} In this section we discuss the case when $\mathcal{D}(\cdot)\in C^{\alpha=1}_+(\mathcal{H}^k_\delta(G;r),\mathbb{C}^{n\times n})$, in this particular cases we will also consider $\norm{}{\cdot}:=\norm*{\mathcal{L}^\infty(G;r)}{\cdot}$, also we will have that in these cases C*-identity is not satisfied in general, nevertheless these cases are also interesting in the analysis of nonlinear evolution equations, see example E.\ref{exe2}, $\mathcal{D}(u)$ is still (SPD) and is also bounded above by $M_D:=\sup_{u\in\mathcal{H}^k_\delta(G;r)}\norm{}{\mathcal{D}(u)}<\infty$. The conditions presented here imply that if $u,v\in dom(A(\cdot))\cap \mathcal{H}^k_\delta(G;r)$ and if we represent $A(u)$ by $A[u]u$ with $A[u]:=a^\dagger\mathcal{D}(u)a$ we will have
\begin{eqnarray}
\norm{}{A[u]}&=&\norm{}{a^\dagger\mathcal{D}(u)a}\leq\norm{}{a}\norm{}{a^\dagger}\norm{}{\mathcal{D}(u)}\leq M_D\mu(A[\mathcal{I}])
\end{eqnarray}
where $\mu(A[\mathcal{I}]):=\norm{}{a}\norm{}{a^\dagger}$ and also
\begin{eqnarray}
\norm{}{A(u)-A(v)}&=&\norm{}{A[u]u-A[v]v}\\
&\leq&\norm{}{A[u]u-A[u]v}+\norm{}{A[u]v-A[v]v}\\
&\leq&\norm{}{A[u]}\norm{}{u-v}+\norm{}{v}\norm{}{A[u]-A[v]}\\
&=&\norm{}{A[u]}\norm{}{u-v}+\norm{}{v}\norm{}{a^\dagger(\mathcal{D}(u)-\mathcal{D}(v))a}\\
&\leq&\norm{}{A[u]}\norm{}{u-v}+\norm{}{v}\norm{}{a}\norm{}{a^\dagger}\norm{}{\mathcal{D}(u)-\mathcal{D}(v)}\\
&\leq&M_D\mu(A[\mathcal{I}])\norm{}{u-v}+c_Dr\mu(A[\mathcal{I}])\norm{}{u-v}\\
&=&(M_D+c_Dr)\mu(A[\mathcal{I}])\norm{}{u-v}
\end{eqnarray}
from these we get the following estimate
\begin{eqnarray}
\norm{}{g(u)-g(v)}&\leq&\norm{}{A(u)-A(v)}+\norm{}{f(u)-f(v)}\\
&\leq&(M_D+c_Dr)\mu(A[\mathcal{I}])\norm{}{u-v}+c_f\norm{}{u-v}\\
&=&((M_D+c_Dr)(\mu(A[\mathcal{I}])+c_f)\norm{}{u-v}.\label{proof_l}
\end{eqnarray}
Hence $g\in C^{\alpha=1}(\mathcal{L}^\infty(G;r))$. The last results can be sumarized in the following.

\begin{theorem} \label{theorem_2}
Nonlinear evolution equations of the form \eqref{eveq2} with diffusion tensor $\mathcal{D}(\cdot)\in C^{\alpha=1}_+(\mathcal{H}^k_\delta(G;r),\mathbb{C}^{n\times n})$ and with $f\in C^{\alpha=1}(\mathcal{H}^k(G;r))$ have a unique local solution $u(t)\in \mathcal{H}^k_\delta(G;r)$ in the time interval $t\in(-\delta,\delta)$ with
\begin{equation}
\delta=\min\left\{\frac{r}{rM_D\mu(A[\mathcal{I}])+M_f},\frac{1}{(M_D+c_Dr)\mu(A[\mathcal{I}])+c_f}\right\}. 
\end{equation}
\end{theorem}
\begin{proof}
From \eqref{proof_l} it is clear that $g\in C^{\alpha=1}_+(\mathcal{L}^\infty(G;r))$ and since
\begin{eqnarray}
\norm{}{g(v)}&=&\norm{}{A(v)+f(v)}\\
&\leq&\norm{}{A(u)}+\norm{}{f(u)}\\
&\leq&rM_D\mu(A[\mathcal{I}])+\norm{}{f(u)}
\end{eqnarray}
wich implies that
\begin{equation}
M_g=\sup_{v\in\mathcal{H}^k_\delta(G;r)}\norm{}{g(v)}\leq rM_D\mu(A[\mathcal{I}])+M_f
\end{equation}
now taking $c_g:=(M_D+c_Dr)\mu(A[\mathcal{I}])+c_f$ the result follows.
\end{proof}

\begin{corollary}
If $u_{h,n}$ denotes the numerical solution of the particular representation of \eqref{eveq2} after $n$ iterations of the Picard method and if we take $u_\delta:=\lim_{n\to\infty}u_{h,n}$ and if $u$ is the solution to \eqref{eveq2} for $t\in(-\delta,\delta)$ with $\delta$ defined in the theorem T.\ref{theorem_2} then we will have
\begin{equation}
\norm{}{u-u_{h,n}}\leq c_uh^{\nu_m}+\frac{(((M_D+c_Dr)\mu(A[\mathcal{I}])+c_f)\delta)^n}{1-((M_D+c_Dr)\mu(A[\mathcal{I}])+c_f)\delta}\delta(r\mu(A[\mathcal{I}])+M_f)
\end{equation}
\end{corollary}
\begin{proof}
If we take $u_\delta:=P_{m,h}u$ for some particular projector $P_{m,h}\in\mathcal{P}(\mathcal{H}_\delta^1(G;r))$ and from from corollary C.\ref{corollary2} we will have
\begin{eqnarray}
\norm{}{u-u_{h,n}}&\leq&\norm{}{u-u_\delta}+\norm{}{u_\delta-u_{h,n}}\\
&\leq&c_uh^{\nu_m}\nonumber \\
&&+\frac{(((M_D+c_Dr)\mu(A[\mathcal{I}])+c_f)\delta)^n}{1-((M_D+c_Dr)\mu(A[\mathcal{I}])+c_f)\delta}\delta(r\mu(A[\mathcal{I}])+M_f)
\end{eqnarray}
where $\nu_m$ is the projection approximation order. 
\end{proof}

\subsubsection{Examples} In this section we will present some examples of implementation of the techniques presented before in the estimation of numerical solutions to the particular forms of some equations of mathematical physics.

\begin{example}\label{exe1}
Given a nonlinear Schrodinger Equation of the form
\begin{equation}
\left
\{
\begin{array}{l}
-\Delta u+|u|^2u=i\partial_tu, x\in G\\
u=0, x\in\partial G\\
u(x,0)=u_0, u_0\in \mathcal{H}(G;r)
\end{array}
\right .
\label{eqex2}
\end{equation}
where $\mathcal{H}(G;r)$ is a discretizable Hilbert space, if we take the particular representation of \eqref{eqex2} with respect to some $P_{m,h}\in\mathcal{P}(\mathcal{H}(G;r))$ then we will obtain
\begin{equation}
\left \{
\begin{array}{l}
u_{h}'(t)=i[\mathcal{M}_{m,h}^{-1}a_{m,h}^\ast\mathcal{M}_{m,h}a_{m,h}]u_{h}+i|u_h|^2u_h\\
u_h(0)=p^\dagger_{m,h}u_0
\end{array}
\label{eqex3}
\right .
\end{equation}
it can be seen that
\begin{equation}
\norm*{\mathcal{H}_{m,h}(G;r)}{|u_h|^2u_h-|v_h|^2v_h}\leq 3r^2\mu(G)\norm*{\mathcal{H}_{m,h}(G;r)}{u_h-v_h}
\end{equation}
and
\begin{equation}
\norm{}{A_{m,h}}\leq\norm{}{\mathcal{M}_{m,h}^{-1}a_{m,h}^\ast\mathcal{M}_{m,h}a_{m,h}}\leq\kappa(\mathcal{M}_{m,h})\norm{}{a_{m,h}}^2\leq\frac{K_a}{h^2}
\end{equation}
wich implies that
\begin{eqnarray}
\norm*{\mathcal{H}(G;r)}{A_{m,h}u_h-A_{m,h}v_h}&\leq&\norm{}{A_{m,h}}\norm*{\mathcal{H}(G;r)}{u_h-v_h}\\
&\leq&\frac{K_a}{h^2}\norm*{\mathcal{H}(G;r)}{u_h-v_h}
\end{eqnarray}
then by T.\ref{picard} we will have that \eqref{eqex3} admits a unique local solution $u_{m,h}(t)$ for $t\in(-\delta,\delta)$ with
\begin{equation}
t=\min\left\{\frac{1}{K_a/h^2+3r^2\mu(G)},\frac{1}{K_a/h^2+r^2}\right\}
\end{equation}
also by C.\ref{corollary2} we will have
\begin{equation}
\norm{}{u-u^k_{m,h}}\leq c_uh^{\nu_m}+\frac{((K_a/h^2+3r^2\mu(G))\delta)^k}{1-(K_a/h^2+3r^2\mu(G))\delta}(K_a/h^2+r^2)\delta.
\end{equation}
\end{example}

\begin{example}\label{exe2}
Given the nonlinear diffusion equation
\begin{equation}
\left\{
\begin{array}{l}
\partial_tu=\partial_x(u^2\partial_xu)\\
u(x,0)=u_0\in\mathcal{H}(G;r)\\
u(0,t)=u(1,t)=0
\end{array}
\right .
\label{exeq5}
\end{equation}
where $\mathcal{H}(G;r)$ is a discretizable Hilbert space, one can obtain its particular representation with respect to some $P_{m,h}\in\mathcal{P}(\mathcal{H}(G;r))$ in the form
\begin{equation}
\left\{
\begin{array}{l}
u'_h(t)=[\mathcal{M}_{m,h}^{-1}d^\dagger_{m,h}\mathcal{M}_{m,h}(u_h^2)d_{m,h}]u_h\\
u_h(0)=p^\dagger_{m,h}u_0
\label{exeq7}
\end{array}
\right .
\end{equation}
we have seen before that $f(v)=v^2\in C^{\alpha=1}_+(\mathcal{H}(G;r))$ wich leads to the estimates
\begin{eqnarray}
\norm{}{A_{m,h}[u_h]}&=&\norm{}{\mathcal{M}_{m,h}^{-1}d^\dagger_{m,h}\mathcal{M}_{m,h}(u_h^2)d_{m,h}}\\
&\leq&\kappa({\mathcal{M}_{m,h}})\norm{}{d_{m,h}}\norm{}{d^\dagger_{m,h}}r^2\norm{}{\mathbf{1}}\\
&\leq&r^2\frac{K_a}{h^2}
\end{eqnarray}
on the other hand we will have by \eqref{proof_l}
\begin{equation}
\norm{}{A_{m,h}[u_h]u_h-A_{m,h}[v_h]v_h}\leq3r^2\frac{K_a}{h^2}\norm{}{u_h-v_h}
\end{equation}hence by T.\ref{theorem_2} exists a unique local solution $u_h(t)$ of \eqref{exeq7} for $t\in(-\delta,\delta)$ with
\begin{equation}
\delta=\min\left\{\frac{h^2}{3r^2K_a},\frac{h^2}{r^2K_a}\right\}=\frac{h^2}{3r^2K_a}
\end{equation}
moreover by C.\ref{corollary2} we will also have that a solution to \eqref{exeq5} will satisfy the estimate
\begin{equation}
\norm{}{u-u_h^k}\leq h^{\nu_m}+\frac{(3r^2K_a\delta/h^2)^k}{1-(3r^2K_a\delta/h^2)}\frac{r^3K_a\delta}{h^2}
\end{equation}
\end{example}

\begin{example} For a vorticity transport problem in the vorticity-stream function approach that has the form
\begin{equation}
\left\{
\begin{array}{l}
\partial_tu=\mathcal{J}\nabla v\cdot\nabla u+\frac{1}{Re}\Delta u\\
\Delta v=-u\\
u(x,0)=u_0, u_0\in\mathcal{H}(G;r)\\
B_1u=u_b, x\in \partial G
B_2v=v_b, x\in \partial G
\end{array}
\right .
\label{vortex}
\end{equation}
where $Re$ is a fixed Reynolds number and where $\mathcal{J}\in\mathbb{R}^{2\times 2}$ is given by
\begin{equation}
\mathcal{J}=\left(\begin{array}{cc}
    0 & -1\\
    1 & 0
  \end{array}\right)
\end{equation}
since we assume $\mathcal{H}(G;r)$ is factorizable we can obtain a particular representation of \eqref{vortex} with respect to some $P_{m,h}\in \mathcal{P}(\mathcal{H}(G;r))$ given by
\begin{equation}
\left\{
\begin{array}{l}
u'_h(t)=[\mathcal{J}a_{m,h} \mathcal{G}_{m,h} u_h\cdot a_{m,h}]u_h+\frac{1}{Re}[\mathcal{M}^{-1}_{m,h}\mathcal{M}_{m,h}[a_{m,h}](a_{m,h})] u_h\\
u_h(0)=p_{m,h}^\dagger u_0\\
\end{array}
\right .
\label{vortexp}
\end{equation}
taking $B_{m,h}[v],A_{m,h}\in\mathcal{H}^\ast_{m,h}(G;r)$ to be defined in the form
\begin{eqnarray}
B_{m,h}[v]&:=&\mathcal{J}a_{m,h} \mathcal{G}_{m,h} v\cdot a_{m,h}\\
A_{m,h}&:=&\frac{1}{Re}\mathcal{M}^{-1}_{m,h}\mathcal{M}_{m,h} [a_{m,h}](a_{m,h})
\end{eqnarray}
we can rewrite \eqref{vortexp} in the form
\begin{equation}
\left\{
\begin{array}{l}
u'_h(t)=B_{m,h}[u_h]u_h+A_{m,h}u_h\\
u_h(0)=p_{m,h}^\dagger u_0\\
\end{array}
\right .
\end{equation}
and taking $g(v):=B_{m,h}[v]v+A_{m,h}v$ we can estimate $M_g$, first we will estimate $\norm{}{B_{m,h}[v]}$
\begin{eqnarray}
\norm{}{B_{m,h}[v]}&=&\norm{}{\mathcal{J}a_{m,h} \mathcal{G}_{m,h} v\cdot a_{m,h}}\\
&\leq&\norm{}{\mathcal{G}_{m,h}}\norm{}{a_{m,h}}^2\norm{}{v}\\
&\leq&\frac{K_ar}{mh^2}
\end{eqnarray}
so $M_g$ will satisfy
\begin{equation}
M_g\leq\frac{K_ar}{h^2}(\frac{r}{m}+\frac{1}{Re}\kappa(\mathcal{M}_{m,h}))
\end{equation}
on the other hand we will have that
\begin{eqnarray}
\norm{}{g(u)-g(v)}&\leq&\norm{}{B_{m,h}[u]u-B_{m,h}[u]v}+\norm{}{B_{m,h}[u]v-B_{m,h}[v]v}\nonumber \\
&&+\norm{}{A_{m,h}(u-v)}\\
&\leq&2\frac{K_ar}{mh^2}\norm{}{u-v}+\frac{K_a}{Reh^2}\kappa(\mathcal{M}_{m,h})\norm{}{u-v}\\
&=&\frac{K_a}{h^2}(\frac{2r}{m}+\frac{1}{Re}\kappa(\mathcal{M}_{m,h}))\norm{}{u-v}
\end{eqnarray}
so it can be seen that $g\in C^{\alpha=1}(\mathcal{H}^\ast_{m,h}(G;r))$ and we can use T.\ref{picard} and C.\ref{corollary2} to perform similar estimates to the performed in examples E.\ref{exe1} and E.\ref{exe2}.
\end{example}

\pagebreak

\begin{center}
 \normalsize\scshape{Acknowledgements}
\end{center}

\medskip

Thanks: To God... he does know why, to Mirna, for her love and support, Special thanks to Stanly Steinberg for his support, advice and for taking time to read the first drafts of this research, to Concepci{\'o}n Ferrufino, Rosibel Pacheco, Jorge Destephen, Adalid Guti{\'e}rrez, Eduardo Bravo and Francisco Figeac for their support and advice.

\end{document}